\documentclass[12pt]{amsart}

\usepackage{amscd,amssymb,amsfonts,amsmath,amstext,amsthm}
\usepackage[usenames]{color} 

\newtheorem{thm}{Theorem}[section]
\newtheorem{df}[thm]{Definition}
\newtheorem{remark}[thm]{Remark}
\newtheorem{ex}[thm]{Example}
\newtheorem{cor}[thm]{Corollary}
\newtheorem{prop}[thm]{Proposition}
\newtheorem{lem}[thm]{Lemma}

\numberwithin{equation}{section}

\newcommand{\ot}{\otimes}

\newcommand{\CC}{\mathcal{C}}
\newcommand\CF{\mathcal{F}}

\newcommand\CT{\mathcal{T}}

\newcommand{\BC}{\mathbb{C}}
\newcommand{\BN}{\mathbb{N}}
\newcommand{\BQ}{\mathbb{Q}}
\newcommand{\BR}{\mathbb{R}}
\newcommand{\BZ}{\mathbb{Z}}
\renewcommand\k{\Bbbk}

\newcommand{\eps}{\varepsilon}

\renewcommand{\a}{\alpha}
\renewcommand{\b}{\beta}
\newcommand\G{\Gamma}
\newcommand{\om}{\omega}

\def\Hom{{\mbox{\rm Hom}}}

\newcommand\id{\operatorname{id}}
\newcommand\Tr{\operatorname{Tr}}
\newcommand\tr{\operatorname{tr}}
\newcommand\qexp{\operatorname{qexp}}

\newcommand{\ds}{\displaystyle}
\newcommand\inv{^{-1}}

\renewcommand\o{\otimes}

\newcommand\Mod[1]{{#1\mbox{-\bf{mod}}_{\mathsf{fin}}}}
\renewcommand\mod[1]{{#1\mbox{-\bf{mod}}}}
\newcommand\lsub[1]{{}_{#1}\!}
\newcommand\ol{\overline}
\newcommand\qdim{\operatorname{qdim}_\ell}
\newcommand\qdimr{\operatorname{qdim}_r}
\newcommand\ann{\operatorname{ann}}
\newcommand\ev{\operatorname{ev}}
\newcommand\db{\operatorname{db}}
\newcommand\GL{\operatorname{GL}}
\newcommand\du{^\vee}
\newcommand\bidu{^{\vee\vee}}
\newcommand\bb{\mathbf{b}}

\newcommand\Irr{\operatorname{Irr}}
\begin{document}

\title{On the trace of the antipode and higher indicators}

\author{Yevgenia Kashina}
\address{Department of Mathematics and Statistics\\ Depaul University}
\email{ykashina@condor.depaul.edu}

\author{Susan Montgomery}\address{Department of Mathematics
\\ University of Southern California}\email{smontgom@math.usc.edu}
\thanks{The second author was supported by NFS grant DMS 07-01291}
\author{Siu-Hung Ng}
\address{Department of Mathematics and Statistics\\ Iowa State University}
\email{rng@iastate.edu}
\thanks{The third author was supported by NSA grant H98230-08-1-0078}
\begin{abstract}
We introduce two kinds of gauge invariants for any finite-dimensional Hopf algebra $H$. When $H$ is semisimple over $\BC$, these invariants are respectively, the trace of the map induced by the antipode on
the endomorphism ring of a self-dual simple module, and the higher Frobenius-Schur indicators of the regular representation. We further study the values of these higher indicators in the context of complex semisimple quasi-Hopf algebras $H$. We prove  that these indicators are non-negative provided the module category over $H$ is modular, and that for a prime $p$, the $p$-th indicator is equal to 1 if, and only if,  $p$ is a factor of $\dim H$. As an application, we show  the existence of a non-trivial self-dual simple $H$-module with bounded   dimension which is determined by the value of the second indicator.
\end{abstract}

\maketitle
\setcounter{section}{0}
\section*{Introduction}
Given a Hopf algebra with certain additional structures such as braiding or ribbon, one can define some quantum invariants for knots, links or 3-manifolds (cf. \cite{Turaev}). These topological invariants are indeed determined by the monoidal structure of the representation category of the underlying Hopf algebra. In general, two Hopf algebras with inequivalent monoidal categories of their representations may yield two different quantum invariants. It is then very natural to ask when two finite-dimensional (quasi-) Hopf algebras have equivalent monoidal structures on their representation categories.

    The question was first addressed in \cite{Sch96} for the comodule categories of Hopf algebras in terms of Hopf bi-Galois extensions. The same question for the representation categories of finite-dimensional quasi-Hopf algebras was studied in \cite{EG1} and \cite{NS08}. The representation categories of quasi-Hopf algebras $H$ and $K$ are monoidally equivalent if, and only if, there exists a gauge transformation $F$ on $K$ such that $H$ is isomorphic, as quasi-bialgebra, to the twist $K^F$ of $K$ by $F$. However, to show the existence or non-existence of such gauge transformations and isomorphisms remains highly non-trivial. Therefore, this characterization may not be very practical for determining the gauge equivalence of two given quasi-Hopf algebras.

    Frobenius-Schur (FS) indicators for representations of finite groups were discovered more than a century ago. The notion was recently extended to rational conformal field theory \cite{Bantay97}, to certain $C^*$-fusion categories \cite{FGSV99}, to semisimple Hopf algebras \cite{LM}, to quasi-Hopf algebras \cite{MaN}, and more generally to pivotal categories \cite{NS07b}. It was shown in \cite{MaN} that the second indicators are invariants of the monoidal structure of the representation category of a semisimple quasi-Hopf algebra, and the same result for all these indicators was established in \cite{NS08} and \cite{NS07b}.

    One of the aims of this paper is to obtain invariants for the monoidal category of representations of any  finite-dimensional Hopf algebra $H$. In short, these invariants of $H$ are called \emph{gauge invariants}. Obviously, the dimension or the quasi-exponent (cf. \cite{EG2}) of a finite-dimensional Hopf algebra over $\BC$ are gauge invariants.  We approach this question by re-examining some formulae and properties of FS indicators for semisimple Hopf algebras over $\BC$.

    In \cite{KSZ2}, the $n$-th FS indicator of the regular representation of a semisimple Hopf algebra $H$ over $\BC$ with antipode $S$ is given by
    $$
    \nu_n(H)=\Tr(S \circ P_{n-1})
    $$
    where $P_n$ is the $n$-th Sweedler power map defined in Definition \ref{def:ind}. However, this expression is well-defined for any finite-dimensional Hopf algebra $H$, and we prove in Theorem \ref{t:nu_n} that $\nu_n(H)$ is a gauge invariant of $H$ for all natural number $n$. Moreover, we establish in Proposition \ref{p:recursive} that the sequence $\{\nu_n(H)\}_{n \in \BN}$ of these indicators  is linearly recursive, and hence the minimal polynomial for this sequence is also a gauge invariant of $H$. The second indicator $\nu_2(H)$ is simply the ordinary trace of $S$.  This can be computed quite handily for Taft algebras and they are sufficient to distinguish their gauge equivalence classes as shown in Section \ref{taft}.

    The trace of $S$ also induces another gauge invariant. For a self-dual absolutely simple $H$-module $V$, the antipode induces an automorphism $S_V$ on $H/\ann V$. We proved in Section \ref{TrS} that $\Tr(S_V)$ is also a gauge invariant of $H$. When $H$ is a complex semisimple Hopf algebra, $\Tr(S_V)$ is equal to the product $\dim V \cdot \nu_2(V)$ of the dimension and the second FS indicator of $V$ \cite{LM}. In general, if $H$ is pivotal, we prove here in Proposition \ref{p:product} that
    \begin{equation}\label{eq:01}
     \Tr(S_V) = \qdim V \cdot \nu_2(V)
    \end{equation}
    where $\qdim V$ and $\nu_2(V)$ are respectively the quantum dimension and the 2nd indicator of $V$ relative to any pivotal element of $H$.

     In the semisimple case, FS indicators have also been studied extensively in many different contexts (cf. \cite{KMM}, \cite{KSZ2}, \cite{Sch04}, \cite{NS07a}). The explorations carried out in those articles discovered some interesting relations between the dimension of the algebra, or the category, and the values of these indicators.

     We learned in group theory that if a finite group $G$ admits a self-dual complex irreducible representation, then $|G|$ must be even. In terms of FS indicators, this means if $G$ admits an irreducible representation with non-zero second indicator, then $|G|$ is even. This result was generalized to semisimple Hopf algebras over an algebraically closed field of characteristic 0 in \cite{KSZ}. In this paper, we extend this result in Theorem \ref{t:primeind} to any semisimple quasi-Hopf algebra over $\BC$ at any prime degree  FS indicator.

     The value of $\nu_2(H)$ or $\Tr(S)$ is of particular interest. In the group case, it counts the number of involutions in the group and hence it is always positive. We show by an example in Section \ref{values} that this is not necessarily the case for Hopf algebras. However, it is worth noting that if the representation category of a semisimple complex quasi-Hopf algebra $H$ is modular, then the FS indicator $\nu_2(H)$ is a non-negative integer. In particular, if $H$ is a semisimple factorizable complex Hopf algebra, then $\Tr(S)$ is a non-negative integer and it has the same parity as $\dim H$. The question of whether $\Tr(S)=0$ for some semisimple complex Hopf algebra is not settled in this paper.

     The actual value $\Tr(S)$ of a complex even dimensional semisimple Hopf algebra $H$ can also be used to determine the existence of a non-trivial self-dual irreducible representation $V$ of $H$ such that
     $$
     \dim V \le \frac{\dim H-1}{|\Tr(S)-1|}\,.
     $$
     Moreover, $\nu_2(V)=1$ if $\Tr(S) >1$, and $\nu_2(V)=-1$ if $\Tr(S)<1$. This  result is proved in Theorem \ref{bound}.

     The paper is organized as follows: We collect some basic definitions, conventions and results in Section 1. In Section 2, we introduce the gauge invariant $\nu_n(H)$ for any finite-dimensional Hopf algebra $H$, and study the sequence $\{\nu_n(H)\}_{n \in \BN}$ of these invariants. In particular, $\nu_2(H)$ is the trace of the antipode. We compute the invariant $\nu_2(T)$ for all Taft algebras $T$, and the sequence $\{\nu_n(T)\}$ for $T$ of dimensions 4 and 9 in Section 3. We continue to introduce another gauge invariant $\gamma(V)$ for all absolutely simple $H$-modules $V$ in Section 4, and we show that $\gamma(V)$ is the product of the quantum dimension and the 2nd Frobenius-Schur of $V$ when $H$ is a pivotal Hopf algebra \eqref{eq:01}. In Section 5, we  study  the values of $\nu_n(H)$ and the positivity of $\nu_2(H)$ for a complex semisimple (quasi-) Hopf algebra $H$ in relation to its dimension. Finally, in Section 6, we show an application of $\nu_2(H)$ in determining the existence of certain self-dual simple submodules of a semisimple Hopf algebra $H$ over $\BC$.

\section{Preliminaries}
In this section, we review some basic facts on gauge equivalence of Hopf algebras, and define some terminology and conventions that we will use in the paper. Throughout this paper, the tensor product of two vector spaces $V, W$ over the base field $\k$ will be denoted by $V \ot W$.

Let $H$ be a Hopf algebra over the field $\k$ with counit $\eps$, comultiplication $\Delta$ and antipode $S$. We will use the Sweedler notation $\Delta(h)=h_1 \o h_2$ for  $h \in H$ with the summation suppressed.  Following \cite{Kas}, a \emph{gauge transformation} of a Hopf algebra $H$ is an invertible element $F \in H^{\ot 2}$ such that $(\eps \ot \id)(F) = (\id \ot \eps)(F)=1$. One can \emph{twist} the Hopf algebra $H$ with a gauge transformation $F$ to a quasi-Hopf algebra $H^F$ which is the same algebra $H$ with the same counit $\eps$ but its \emph{comultiplication} $\Delta^F$ is given by
$$
\Delta_F(h) = F \Delta(h) F\inv \quad \text{for } h \in H\,.
$$
The quasi-Hopf algebra $H^F$ is an \emph{ordinary} Hopf algebra if, and only if, the gauge transformation $F$ satisfies
\begin{equation}\label{eq:pF}
1 \ot 1 = \partial F=(1 \ot F)(\id \ot \Delta)(F)(\Delta \ot \id)(F\inv)(F\inv \ot 1).
\end{equation}
In general, we call a gauge transformation $F$ of a Hopf algebra $H$ a 2-cocycle if $F$ satisfies \eqref{eq:pF}.

If $F = \sum_{i} f_i \ot g_i$ is a 2-cocycle of $H$ with inverse $F\inv = \sum_i d_i \ot e_i$, then \eqref{eq:pF} implies
\begin{equation}\label{eq:ab2}
 \a_F\b_F=1, \quad \text{and}\quad \b_F\a_F=1\,.
\end{equation}
where
\begin{equation} \label{eq:ab1}
 \a_F=\sum_j S(d_j)e_j \quad\text{and} \quad \b_F = \sum_i f_iS(g_i)\,.
\end{equation}
Moreover, $H^F$ is a Hopf algebra with its antipode $S^F$ defined by
\begin{equation}\label{eq:S_F}
S^F(h)= \b_F S(h) \a_F
\end{equation}
 The Hopf algebra  $H^F$ is called a  (\emph{Drinfeld}) \emph{twist} of the Hopf algebra $H$
by the 2-cocycle $F$. It is worth noting that \eqref{eq:ab2} does not hold in general if $F$ is
not a 2-cocycle.

Two Hopf algebras $H$ and $H'$ over a field $\k$ are said to be gauge equivalent if there exists a 2-cocycle $F$ of $H$ such that $H'\stackrel{\sigma}{\cong} H^F$ are isomorphic as bialgebras. The isomorphism $\sigma$ induces a $\k$-linear equivalence $\lsub \sigma (-): \mod{H} \to \mod{H'}$ as follows: for $V \in \mod{H}$, $\lsub \sigma V=V$ as $\k$-linear space with the $H'$-action given by
$$
h' v := \sigma(h')v \quad \text{for all $h' \in H'$ and $v \in V$,}
$$
and $\lsub\sigma(f) = f$ for any map $f$ in $\mod{H}$. The 2-cocycle $F$ defines the natural isomorphism
$$
\xi:= \left(\lsub\sigma V \o \lsub\sigma W \xrightarrow{F\cdot} \lsub \sigma (V \o W)\right)
$$
for any $V, W \in \mod{H}$. The triple $(\lsub \sigma(-), \xi, \id_\k )$ is an equivalence of tensor categories.

Conversely, if $H, H'$ are finite-dimensional such that $\mod{H}$ and $\mod{H'}$ are equivalent tensor categories, then it follows by \cite{Sch96} (or more generally \cite{EG1}, \cite{NS08}) $H$ and $H'$ are gauge equivalent. We summarize this Hopf algebra version of \cite[Theorem 2.2]{NS08} as follows.
\begin{thm}\label{t:p1}
Two finite-dimensional Hopf algebras $H$ and $H'$ over a field $\k$ are gauge equivalent if, and only if, $\Mod{H}$ and $\Mod{H'}$ are equivalent as ($\k$-linear) tensor categories. Moreover,  if a $\k$-linear functor $\CF: \mod{H} \to \mod{H'}$ defines an equivalence of tensor categories, then there exist a 2-cocycle $F$ of $H$ and a bialgebra isomorphism $\sigma: H' \to H^F$ such that $\k$-linear equivalences $\CF$ and $\lsub\sigma(-)$ are naturally isomorphic.\qed
\end{thm}

A quantity $f(H)$ defined for each Hopf algebra $H$ is called a \emph{gauge invariant} if $f(H)=f(H')$ for all Hopf algebras $H'$ which are gauge equivalent to $H$. We will introduce and discuss some gauge invariants in the remainder of this paper.

\section{Invariance of the $n$-th indicator $\nu_n(H)$}
In this section, we introduce a sequence of scalars $\nu_n(H)$ for each finite-dimensional Hopf algebra $H$ over a field $\k$. We obtain a formula for $\nu_n(H)$ in terms of the integrals of $H$ and $H^*$. Moreover, we prove that this sequence $\{\nu_n(H)\}_{n \in \BN}$ is linearly recursive and  is a gauge invariant of $H$. In particular, the \emph{minimal polynomial} $p_H$ of the sequence of higher indicators of $H$ is also a gauge invariant.

Let $H$ be a finite-dimensional Hopf algebra over a field $\k$ with antipode $S$, multiplication $m$, comultiplication $\Delta$ and counit $\eps$. We define
$$
 \Delta^{(1)}=\id_H, \quad \text{and} \quad
\Delta^{(n+1)} = (\id \o \Delta^{(n)})\circ \Delta
$$
for all integers $n \ge 1$. By the coassociativity of $\Delta$, it is well-known that
$$
\Delta^{(n+1)} = ( \Delta^{(n)} \o \id)\circ \Delta
$$
for all integers $n \ge 1$. Similarly, we let $m^{(1)}=\id_H$,\,$m^{(n)}: H^{\o n} \to H$ be the multiplication map of $H$.
Note that $\Delta^{(n+1)}$ defined here is equal to $\Delta_n$ in \cite[p 11]{Sw}.
\begin{df}\label{def:ind}
The \textbf{$n$-th Sweedler power} $P_n(h)=h^{[n]}$ of an element $h \in H$ is defined as
$$
P_n(h) = m^{(n)} \circ \Delta^{(n)}(h) \quad\text{for }n \ge 1,
$$
and we set $P_0(h)= h^{[0]}=\eps(h) 1$.
We define
$$
\nu_n(H) := \Tr(S \circ P_{n-1})
$$
for each positive integer $n$, and call it the \textbf{$n$-th indicator} of $H$.
\end{df}

Obviously, $\nu_1(H)=1$, and $\nu_2(H)=\Tr(S)$. When $H$ is semisimple and $\k$ is an algebraically closed field of characteristic zero, the $n$-th Frobenius-Schur indicator $\nu_n(V)$ is defined for each $V \in \mod{H}$. The scalar $\nu_n(H)$ coincides with the $n$-th Frobenius-Schur indicator of the regular representation of $H$ (cf. \cite{KSZ2}). It follows by the works of \cite{NS08} and \cite{MaN} that if $H'$ is a Hopf algebra and $\CF: \Mod{H} \to \Mod{H'}$ defines an equivalence of tensor categories, then
$\nu_n(V) = \nu_n(\CF(V))$ for all positive integer $n$ and $V \in \mod{H}$. Since $\CF(H)\cong H'$ as $H'$-modules, we have  $\nu_n(H) = \nu_n(H')$ for all positive integers $n$. Therefore,
the sequence $\{\nu_n(H)\}_{n \in \BN}$ is a gauge invariant of  \emph{semisimple} Hopf algebras over $\k$.

The notion of higher Frobenius-Schur indicators for a module over a \emph{general} finite-dimensional Hopf algebra $H$ remains unclear. Nevertheless, the following theorem shows the sequence $\{\nu_n(H)\}_{n \in \BN}$ is a gauge invariant.

\begin{thm}\label{t:nu_n}
 Let $H$ a finite-dimensional Hopf algebra over a field $\k$. Then the sequence $\{\nu_n(H)\}_{n \in \BN}$ is an invariant of the gauge equivalence class of Hopf algebras of $H$. If $\lambda \in H^*$ is a right integral and $\Lambda \in H$ is a left integral such that $\lambda(\Lambda)=1$, then
 $$
 \nu_n(H) = \lambda(S(\Lambda^{[n]}))
 $$
 for all positive integer $n$.
\end{thm}

To prove the theorem, we need to establish a relationship between $\Delta_F^{(n)}$ and $\Delta^{(n)}$ for a given 2-cocycle $F$ of $H$.  Let us define
$$
F_1=1_H,\quad \text{and}\quad F_{n+1} = (1 \o F_n)(\id \o \Delta^{(n)})(F)
$$
for all integers $n \ge 1$. In particular, $F_2 = (1 \o F_1)(\id \o \Delta^{(1)})(F) =F$.
\begin{lem} \label{l:F_n}
Let $F$ be a 2-cocycle of a finite-dimensional Hopf algebra over $\k$. Then the following equations hold for all positive integers $n$:
\begin{equation}\label{eq:F_n}
  F_{n+1} = (F_n \o 1)(\Delta^{(n)} \o \id)(F),
\end{equation}
\begin{equation} \label{eq:Delta_F^n}
  F_n \Delta^{(n)}(h) = \Delta_F^{(n)}(h) F_n \quad \text{for all } h \in H\,.
\end{equation}
\end{lem}
\begin{proof}
  Both equations obviously hold for $n=1,2$, and we proceed to prove the equalities by induction on $n$.
  Assume the induction hypothesis \eqref{eq:Delta_F^n}.
  Let $F = \sum_i f_i \ot g_i$.
  Then, for $h \in H$,
  $$
  \sum_{(h), i} f_i h_1 \o g_i h_2 = \Delta_F(h)F
  $$
  and
  \begin{multline*}
   (\id \ot \Delta_F^{(n)})(F)(1 \ot F_n) = \sum_i f_i \o \Delta_F^{(n)}(g_i)F_n  \\ =
   \sum_i f_i \o F_n \Delta^{(n)}(g_i) = (1 \ot F_n) (\id \o \Delta^{(n)})(F) =F_{n+1}\,.
  \end{multline*}

  Thus,
  \begin{multline*}
    F_{n+1} \Delta^{(n+1)}(h) = \sum_{(h), i}(1 \o F_n )(f_i h_1 \ot \Delta^{(n)}(g_i h_2)) \\
    = \sum_{(h), i}(f_i h_1 \ot \Delta_F^{(n)}(g_i h_2))(1 \o F_n )
    = \left((\id \ot \Delta_F^{(n)})(\Delta_F(h)F)\right)(1 \o F_n ) \\
     = \Delta_F^{(n+1)}(h)(\id \ot \Delta_F^{(n)})(F)(1 \o F_n )=
      \Delta_F^{(n+1)}(h) F_{n+1}\,.
  \end{multline*}
 Using the induction assumption, we find
  \begin{multline*}
   F_{n+2} = (1 \o F_{n+1}) (\id \o \Delta^{(n+1)})(F) \\
   =   (1 \o F_n \o  1)\left((\id \o \Delta^{(n)}\o \id)(1 \o F)\right) (\id \o \Delta^{(n+1)})(F)\\
   =   (1 \o F_n \o  1)\left((\id \o \Delta^{(n)}\o \id)(1 \o F)\right) \left((\id \o  \Delta^{(n)} \o \id)(\id \o \Delta)(F)\right)\\
   =   (1 \o F_n \o  1)(\id \o \Delta^{(n)}\o \id)((1 \o F) (\id \o \Delta)(F))\,.
  \end{multline*}
  Here, the third equality is a consequence of the coassociativity of $\Delta$, and the last equality follows from the fact that $\Delta$ is an algebra map. By \eqref{eq:pF}, we have
  \begin{multline*}
   F_{n+2} = (1 \o F_n \o  1)(\id \o \Delta^{(n)}\o \id)((F \o 1) (\Delta\o \id)(F))\\
   = (F_{n+1} \o 1)(\id \o \Delta^{(n)}\o \id)(\Delta\o \id)(F)
   = (F_{n+1} \o 1)(\Delta^{(n+1)}\o \id)(F) \,.\qedhere
  \end{multline*}
\end{proof}

\begin{remark}
Lemma \ref{l:F_n} can also be obtained by considering the proof of a coherence result of Epstein \cite{E}. For the sake of completeness, a direct algebraic proof is provided for the lemma.
\end{remark}

\begin{lem}\label{l:SweedlerPower}
  Let $H$ be a finite-dimensional Hopf algebra over $\k$ and $F$ a 2-cocycle of $H$.  Then for any positive integer $n$ and $h \in H$, we have
  $$
  h^{[n+1]} =  m^{(n+1)}\left((\id \o \Delta_F^{(n)} \o \id)  \left((1\o F)(1 \o \Delta(h))(F\inv \o 1)\right)\right)
  $$
  where $m^{(n+1)}$ denotes the multiplication of $H$.
\end{lem}
\begin{proof}
By Lemma \ref{l:F_n} and the coassociativity of $\Delta$, we find
\begin{eqnarray*}
& &(\id \o \Delta_F^{(n)} \o \id)  \left((1\o F)(1 \o \Delta(h))(F\inv \o 1)\right) \\
& = &
(1 \o F_n \o 1)\left((\id \o \Delta^{(n)} \o \id)  \left((1\o F)(1 \o \Delta(h))(F\inv \o 1)\right) \right)(1 \o F_n\inv \o 1) \\
&=& (1 \o F_{n+1})(1 \o \Delta^{(n+1)}(h))(F_{n+1}\inv  \o 1)\,.
\end{eqnarray*}
Note that if $Z \in H^{\o (n+1)}$ is invertible, then
$$
m^{(n+2)}\left((1 \o Z)(x_1 \o \cdots \o x_{n+2})(Z\inv \o 1)\right) =  x_1 \cdots x_{n+2}
$$
for all $x_1, \dots, x_{n+2}\in H$ (cf. \cite[Lemma 4.4]{NS08}). Therefore,
$$
m^{(n+1)}\left((1 \o F_{n+1})(1 \o \Delta^{(n+1)}(h))(F_{n+1}\inv  \o 1)\right) = m^{(n+1)} \circ \Delta^{(n+1)}(h) =h^{[n+1]}
$$
and so the result follows.
\end{proof}
Now, we can prove the main result of this section.
\begin{proof}[Proof of Theorem \ref{t:nu_n}]
Let $H$ be a Hopf algebra over a field $\k$ with antipode $S$. Suppose $\Lambda$ is a left  integral of $H$, and $\lambda$ a right integral of $H^*$ such that $\lambda(\Lambda)=1$. By  Radford's trace formulae (cf. \cite{Rad94}),
$$
\Tr(f) = \lambda(S(\Lambda_2)f(\Lambda_1)) \quad \text{for any $\k$-linear operator $f$ on $H$},
$$
where $\Delta(\Lambda)=\Lambda_1 \o \Lambda_2$ is the Sweedler notation with the summation suppressed.  Thus, we have
$$
\nu_n(H)= \lambda(S(\Lambda_2)S(\Lambda_1^{[n-1]})) = \lambda(S(\Lambda_1^{[n-1]}\Lambda_2))=
\lambda(S(\Lambda^{[n]}))\,.
$$
 Now let $F=\sum_i f_i \o g_i$ be a 2-cocycle of $H$ and $F\inv=\sum_j d_j \o e_j$. Then, it follows by  \eqref{eq:ab2}, \eqref{eq:ab1}  that  $u=\b_F=\sum_i f_i S(g_i)$ is invertible with $u\inv=\a_F=\sum_i S(d_i) e_i$. Moreover, by \eqref{eq:S_F},  the antipode $S^F$ of $H^F$ is given by
 $$
 S^F(h) = uS(h)u\inv\,.
 $$
 Let $P_n^F(h)$ denote the $n$-th Sweedler power of an element $h \in H^F$. Then, for $n \ge 1$,
 \begin{multline*}
   \nu_{n+1}(H^F) = \lambda(S(\Lambda_2)S^F \circ P_n^F(\Lambda_1)) \\
     =\lambda(S(\Lambda_2)u S(P_n^F(\Lambda_1))u\inv)
      =\sum_{i,j}\lambda(S(\Lambda_2) f_iS(g_i) S(P_n^F(\Lambda_1))S(d_j)e_j) \\
 = \sum_{i,j}\lambda(S(S\inv(f_i)\Lambda_2) S(d_j P_n^F(\Lambda_1)g_i) e_j)\,.
 \end{multline*}
Recall from \cite{Rad94} that
$$
a \Lambda_1 \o \Lambda_2 =  \Lambda_1 \o S\inv(a)\Lambda_2, \quad  \Lambda_1a \o \Lambda_2 =
\Lambda_1 \o \Lambda_2 S(a\leftharpoonup \a)
$$
$$
\text{and}\qquad \lambda(ab)=\lambda(S^2(b\leftharpoonup\a)a)\quad\text{for all }a, b \in H
$$
where $\a$ is the distinguished group-like element in $H^*$ defined by $\Lambda a = \Lambda \a(a)$.
Using these properties, we have
\begin{multline*}
\nu_{n+1}(H^F)= \sum_{i,j}\lambda(S(\Lambda_2) S(d_j P_n^F(f_i\Lambda_1)g_i) e_j)  \\ =
\sum_{i,j}\lambda(S^2(e_j \leftharpoonup \a) S(\Lambda_2) S(d_j P_n^F(f_i\Lambda_1)g_i) e_j) \\
= \sum_{i,j}\lambda(S(\Lambda_2 S(e_j \leftharpoonup \a)) S(d_j P_n^F(f_i\Lambda_1)g_i)) \\
= \sum_{i,j}\lambda(S(\Lambda_2) S(d_j P_n^F(f_i\Lambda_1 e_j)g_i)) \\
= \sum_{i,j}\lambda(S(d_j P_n^F(f_i\Lambda_1 e_j)g_i \Lambda_2)) \,.
\end{multline*}
Notice that
$$
\sum_{ij} d_j P_n^F(f_i\Lambda_1 e_j)g_i \Lambda_2 =
m\left((\id \o \Delta_F^{(n)} \o \id)  \left((1\o F)(1 \o \Delta(\Lambda))(F\inv \o 1)\right)\right)\,.
$$
Hence, by Lemma \ref{l:SweedlerPower}, the last expression is equal to $\Lambda^{[n+1]}$. Therefore,
$$
\nu_{n+1}(H^F) = \lambda(S(\Lambda^{[n+1]}))\,.
$$

If $H'$ is a Hopf algebra which is gauge equivalent to $H$, then there exists a 2-cocycle $F$ of $H$ such that $H'\stackrel{\sigma}{\cong} H^F$ as bialgebras. Let $S'$ and $P_n'$ be the antipode and the $n$-th Sweedler power map of $H'$ respectively. Then
$$
\sigma\circ S' \circ P_n' = S^F \circ P^F_n \circ \sigma\,.
$$
Therefore,
$$
\nu_n(H') = \nu_n(H^F) = \nu_n(H)
$$
for all positive integer $n$.
\end{proof}

 Suppose $\lambda \in H^*$ is a right integral and $\Lambda \in H$ is a left integral such that $\lambda(\Lambda)=1$. Then $\lambda_\ell=\lambda \circ S$ is a left integral of $H^*$ and
 $$
 \lambda_\ell(\Lambda)=\lambda(S (\Lambda)) = \lambda(\Lambda)=1 \quad \text{(cf. \cite{Rad94})}.
 $$
 Therefore, $\nu_n(H) = \lambda_\ell( \Lambda^{[n]})$.

 On the other hand, $\Lambda_r = S(\Lambda)$ is a right integral of $H$, and we obviously have
 $$
 \lambda(\Lambda_r)=1, \quad \Lambda_r^{[n]} = S(\Lambda^{[n]})\,.
 $$
 Thus, $\nu_n(H) = \lambda(\Lambda_r^{[n]})$. We summarize this conclusion in the following corollary.
 \begin{cor} \label{c:4.6}
   Let $H$ be a finite-dimensional Hopf algebra over $\k$. Suppose $\lambda \in H^*$ and $\Lambda \in H$ are both left integrals (or both right integrals) such that $\lambda(\Lambda)=1$. Then
   $$
   \nu_n(H) = \lambda(\Lambda^{[n]})
   $$
   for all positive integer $n$. \qed
 \end{cor}

We note, however, that $\nu_n(H)$ is not preserved under \emph{twisting} by more general pseudo-cocycles of $H$. Nikshych \cite{Ni} shows
that the group algebras $\BC Q$ and $\BC D_8$, where $Q$ is the quaternion group and $D_8$ is the
dihedral group of order 8, are twists of each other by a pseudo-cocycle. However neither the
indicators nor $\Tr(S)$ are the same for the two groups.

Recall that a sequence $\{\b_n\}_{n \in \BN}$ in $\k$ is said to be \emph{linearly recursive} if it satisfies a non-zero polynomial $f(x)=f_0+f_1x+\cdots+ f_{m-1}x^{m-1} + f_m x^m \in \k[x]$, i.e.
$$
f_0 a_n+f_1 a_{n+1}+\cdots+ f_{m-1}a_{n+m-1} + f_m a_{n+m} =0 \quad \text{for all } n \in \BN.
$$
The monic polynomial of the least degree satisfied by a linearly recursive sequence is called the \emph{minimal polynomial} of the sequence.

For the case of a finite group $G$,
$$
\nu_n(\BC G) =\#\{g \in G\mid g^n=1\}\,.
$$
Thus, the sequence $\nu_n(\BC G)$ is periodic, and hence linearly recursive as it satisfies the polynomial $x^N-1$ where $N$ is the exponent of $G$. More generally, for any semisimple Hopf algebra $H$ over $\BC$, the sequence $\{\nu_n(H)\}_{n \in \BN}$ is periodic whose period is equal to the exponent of $H$ (cf.  \cite[Proposition 5.3]{NS07a}). However, the example in the following section implies that the sequence of higher indicators $\{\nu_n(H)\}_{n \in \BN}$ is not periodic for an arbitrary finite-dimensional Hopf algebra $H$. Nevertheless, by the following proposition, the sequence is always linearly recursive.

\begin{prop}\label{p:recursive}
Let $H$ be a finite-dimensional Hopf algebra over any field $\k$. Then the sequence $\{\nu_n(H)\}$ is linearly recursive and the degree of its minimal polynomial is at most $(\dim H)^2$.
Also, the minimal polynomial $p_H$ of the sequence of higher indicators is also a gauge invariant.
\end{prop}
\begin{proof}
  Since $\Hom_\k(H, H)$ is of finite dimension, the set of operators $\{P_n\}_{n\ge 0}$ is $\k$-linearly dependent. There exist a positive integer $N \le (\dim H)^2$ and scalars
  $\a_0, \dots, \a_{N-1} \in \k$ such that
  \begin{equation}\label{eq:relation1}
   \a_0 P_0+\cdots \a_{N-1} P_{N-1} +P_N =0.
  \end{equation}
  Recall that $P_n$ is the $n$-th power of $\id_H$ under the convolution product $*$ of $\Hom_\k(H, H)$. Therefore, $P_n*P_m =P_{m+n}$ for any non-negative integers $m,n$. Multiply  Equation \eqref{eq:relation1} by $P_{n-1}$. We find
  \begin{equation}\label{eq:relation2}
   \a_0 P_{n-1}+\cdots \a_{n-2+N} P_{n-2+N} +P_{n-1+N} =0
  \end{equation}
for all positive integers $n$. Apply $S$ to this equation and take trace. We have
$$
\a_0\nu_n(H)+\cdots \a_{n+N-1} \nu_{n-1+N}(H) +\nu_{n+N}(H) =0\,.
$$
Hence, the sequence $\{\nu_n(H)\}$ is linearly recursive and it satisfies a monic polynomial of degree $\le N$.
\end{proof}


\section{Indicators of the Taft algebras}\label{taft}
The Taft algebras are well-known to be non-semisimple. We will compute the 2nd indicators $\nu_2(T)$ of the Taft algebras $T$ as examples in this section. In particular, for the Taft algebra $T_4$ of dimension $4$, we have computed its complete sequence $\{\nu_n(T_4)\}_{n \in \BN}$ which is simply the sequence of positive integers

Let us begin with the definition of the Taft algebras over an algebraically closed field $\k$ of characteristic zero. For a primitive $n$-th root of unity $\om \in \k$, the Taft algebra $T_{n^2}(\om)$ is the $\k$-algebra generated by $g$ and $x$ subject to the relations $g^n=1$, $x^n=0$, and $xg=\om gx$. The Taft algebra is a Hopf algebra with the coalgebra structure and the antipode $S$ given by
$$
\Delta(g)=g\ot g, \quad \Delta(x)=x\ot 1 + g \ot x, \quad \eps(g)=1, \quad \eps(x)=0,
$$
$$
S(g)=g\inv, \quad S(x)=-g\inv x\,.
$$
The subset $\{g^ix^k\mid i,k=0, \dots, n-1\}$ of the Taft algebra $T_{n^2}(\om)$ forms a basis, and hence $\dim T_{n^2}(\om) = n^2$. By induction, one can write down the images of the antipode $S$ on this basis as
\begin{lem}\label{l:5.1} For $i,k = 0, \dots, n-1$, we have
  $$S(g^i x^k) = (-1)^k \om^{-\left(\frac{k(k-1)}{2}+ik\right)} g^{-(i+k)}x^k\,.$$
\end{lem}
By \eqref{def:ind}, $\nu_2(T_{n^2}(\om)) = \Tr(S)$, and so we proceed to compute the trace of $S$ with the following lemma.
\begin{lem}\label{l:2}
  If $S(g^i x^k )= \a g^i x^k$ for some $\a \in \k$, then $k \equiv -2i \pmod{n}$. Moreover,
  $$
  \Tr(S) =
  \left\{
  \begin{array}{ll}
  \displaystyle 2 \sum_{\ell=0}^{\frac{n-1}{2}} \om^\ell= \frac{2}{1+\om^{(n+1)/2}} & \text{if $n$ is odd,}\\\\
  \displaystyle 2 \sum_{\ell=0}^{\frac{n}{2}-1} \om^\ell=\frac{4}{1-\om} & \text{if $n$ is even.}\\
  \end{array}
  \right.
  $$
\end{lem}
\begin{proof}
  The first assertion follows immediately from Lemma \ref{l:5.1}. Now, let $u=\left[\frac{n}{2}\right]$.
  We first assume $n$ is odd. Then $u=\frac{n-1}{2}$ and so $n-2u=1$. In particular,  $u$ is the inverse of $-2$ modulo $n$. Note that
  $\om^{1/2}$ is uniquely determined and is given by
  $$
  \om^{1/2} = \om^{-u}\,.
  $$
  Suppose $g^ix^k$ is an eigenvector of $S$. Then $k \equiv -2i \pmod n$ and hence $i \equiv uk \pmod n$. The eigenvalue $\a$ associated with this eigenvector is
  $$
  \a=(-1)^k \om^{-\left(\frac{k(k-1)}{2}+ik\right)} = (-1)^k \om^{-\left(k(k-1)(-u)+uk^2\right)} = (-1)^k \om^{-uk}\,.
  $$
  Thus,
  $$
  \Tr(S)=\sum_{k=0}^{n-1} (-1)^k \om^{-uk} = \frac{1+\om^{-un}}{1+\om^{-u}}=\frac{2}{1+\om^{-u}}\,.
  $$
  Obviously, $-u \equiv (n+1)/2 \pmod n$. On the other hand, $\{\om^{-uk}\mid k=0,\dots, n-1\}=\{\om^k\mid k=0,\dots, n-1\}$, and so
  $$
  \Tr(S)=\sum_{k \text{ even}}\om^{-uk} - \sum_{k \text{ odd}} \om^{-uk} = \sum_{\ell=0}^u \om^\ell -\sum_{\ell=u+1}^{n-1} \om^\ell = 2\sum_{\ell=0}^u \om^\ell\,.
  $$

  Now, we assume $n$ is even. Then $u=\frac{n}{2}$ and the congruence $k \equiv -2i \pmod n$ implies that $k$ must be even, and so $k=2\ell$  some non-negative integer $\ell < \frac{n}{2}$, and
  $$
  \ell \equiv -i \pmod {\frac{n}{2}}.
  $$
   Consequently,  $i=\frac{n}{2}-\ell$ or $n-\ell$. Thus, the eigenvalues associated with the eigenvectors $g^{\frac{n}{2}-\ell}x^{2\ell}$ and $g^{n-\ell}x^{2\ell}$ are
  $$
  \om^{-\left(\ell(2\ell-1)+ n\ell-2\ell^2\right)} \quad \text{and}\quad
  \om^{-\left(\ell(2\ell-1)+ 2n\ell-2\ell^2\right)}
  $$
  respectively. However, both eigenvalues are equal to $\om^\ell$. Note that $\om^u=-1$. Therefore,
  $$
  \Tr(S)=\sum_{\ell=0}^{u-1} 2\om^\ell = 2\sum_{\ell=0}^{u-1} \om^\ell =
  2 \left(\frac{1-\om^{u}}{1-\om}\right)= \frac{4}{1-\om}\,. \qedhere
  $$
\end{proof}

It has been shown in \cite[Corollary 2.4]{Sch00} that $T_{n^2}(\om)$ is uniquely determined by its associated tensor category $\Mod{T_{n^2}(\om)}$. In particular, $T_{n^2}(\om)$ and $T_{n^2}(\om')$
are not gauge equivalent if $\om \ne \om'$. Here we give an alternative proof by computing the trace of the antipode and using Theorem \ref{t:nu_n}.

\begin{cor}
  Let $\om_1, \om_2 \in \k$ be primitive $n$-th roots of unity. Suppose $S_i$ is the antipode of the Taft algebra $T_{n^2}(\om_i)$, $i=1,2$. Then $\Tr(S_1)=\Tr(S_2)$ iff $\om_1 =\om_2$. In particular,
  $T_{n^2}(\om_1)$ and $T_{n^2}(\om_2)$ are gauge equivalent iff $\om_1=\om_2$.
\end{cor}

\begin{proof}
  Suppose $\Tr(S_1)=\Tr(S_2)$.
  If $n$ is even, then, by Lemma \ref{l:2}, we have $\om_1 =\om_2$. If $n$ is odd, then
  $$
  \om_1^{(n+1)/2}=\om_2^{(n+1)/2}.
  $$
  Hence
  $$
  \om_1= \left(\om_1^{(n+1)/2}\right)^2 = \left(\om_2^{(n+1)/2}\right)^2 = \om_2.
  $$
  The second statement is an immediate consequence of Theorem \ref{t:nu_n}.
\end{proof}

The general formula for the $n$-th indicator of $T_{m^2}(\om)$ is less obvious, but $T_4(-1)$ is an exception.
\begin{ex}
{\rm
The $n$-th indicator of $T_4(-1)$ is $n$, and hence its minimal polynomial is $(x-1)^2$. To show this observation, we can apply Corollary \ref{c:4.6}. In $T_4(-1)$, $\Lambda = x + g x$ is a left integral, and  $\lambda \in {T_4(-1)}^*$ defined by $\lambda(1)=\lambda(g)=\lambda(x)=0$ and $\lambda(gx)=1$ is a left integral of ${T_4(-1)}^*$ such that $\lambda(\Lambda)=1$. By induction, one can show that
$$
\Delta^{(n)}(x) = x \ot 1^{\ot (n-1)} + g \ot x \ot 1^{\ot (n-2)} + \cdots + g^{\ot (n-1)} \ot x \quad\text{for }n \ge 2,$$
where $z^{\ot (\ell)}$ denotes the $\ell$-folded tensor $z \ot \cdots \ot z$.  Thus, using $g^2=1$,
$$
\Delta^{(n)}(gx) = gx \ot g^{\ot (n-1)} + 1 \ot gx \ot g^{\ot (n-2)} + \cdots + 1^{\ot (n-1)} \ot gx \quad\text{for }n \ge 2.$$
Therefore,
$$
\Lambda^{[n]} = \sum_{i=0}^{n-1} g^i x + gx g^{n-1-i} =\sum_{i=0}^{n-1} g^i x + (-1)^{n-1-i} g^{n-i}x =\sum_{i=0}^{n-1} g^i x + (-1)^i g^{i+1}x\,.
$$
Since $\lambda(\Lambda^{[n]})$ is equal to the coefficient of $gx$ in $\Lambda^{[n]}$, we find
$$
\nu_n(T_4(-1))=\lambda(\Lambda^{[n]})=n
$$
for $n \in \BN$.  \qed
}
\end{ex}

A more delicate but direct computation shows that
$$
\nu_n(T_9(\om))=\left\{\begin{array}{ll} n(2+\om) & \text{ if } n \equiv 0 \pmod 3, \\
  n & \text{ if } n \equiv 1 \pmod 3, \\
  n(1+\om) & \text{ if } n \equiv 2 \pmod 3,
\end{array}
\right.
$$
where $\om \in \k$ is a primitive third root of unity. The minimum polynomial of the sequence is $(x-1)^2(x-\om^{-1})^2$.
One can continue the computation using GAP to find the minimal polynomial $p_m(x)$ of $T_{m^2}(\om)$ for $m \ge 3$. We summarize our observation for $m \le 24$ as follows:
$$
p_m(x) = \left\{
\begin{array}{ll}
\dfrac{(x^m-1)^2}{(x-\om)^2} & \text{if } m=2,3, 4, 6, 8, 12, 24,\\ \\
(x^m-1)^2 & \text{otherwise}.
\end{array}
\right.
$$
It has been shown in \cite{EG2} that the quasi-exponent $\qexp(H)$ of a finite-dimensional Hopf algebra $H$  is a gauge invariant, and $\qexp(T_{m^2}(\om))=m$ for all positive integers $m$. The preceding observation suggests some relation between the quasi-exponent $\qexp(H)$ and the minimal polynomial $p_H$ of a finite-dimensional Hopf algebra $H$. It would be interesting to know how they are actually related.

\section{Gauge invariance of $\Tr(S_V)$} \label{TrS}
Let $H$ be a finite-dimensional Hopf algebra over a field $\k$ with antipode $S$.
For $V \in \Mod{H}$, we denote the left dual of $V \in \Mod{H}$ by $V\du$. Suppose $V \in \Mod{H}$ is  self-dual, i.e. $V\cong V\du$ as $H$-modules. Then $S(\ann V)=\ann V$ and so $S$ induces an algebra anti-automorphism $S_V$ on $H/\ann V$ defined by $S_V(h +\ann V)= S(h)+\ann V$. We define for each absolutely simple $V \in \Mod{H}$
\begin{equation} \label{eq:gamma}
\gamma(V)=\left\{
\begin{array}{ll}
\Tr(S_V) & \text{if }V \cong V\du\,,\\
0 & \text{otherwise.}
\end{array}
\right.
\end{equation}
Obviously, $\gamma(V)$ depends only on the isomorphism class of $V$.

In the semisimple case, $\Tr(S_V)$ is closely related to the Frobenius-Schur
indicator. More precisely, $\Tr(S_V)=\nu_2(V)\dim V$. It may also prove to be important in the non-semisimple case. We note that Jedwab has begun the study of $\Tr(S_V)$ in \cite{J}, and has computed $\Tr(S_V)$ for the irreducible representations of $u_q(sl_2)$. Some additional work on this topic has been done in \cite{JK}.

Similar to the case of the 2nd Frobenius-Schur indicator of a simple module over a semisimple complex Hopf algebra, we prove in this section that $\gamma(V)$ is an invariant of the tensor category $\Mod{H}$.
\begin{thm}\label{t:1}
  Let $H$, $H'$ be finite-dimensional Hopf algebras over a field $\k$. If $\CF: \Mod{H} \to \Mod{H'}$ is a equivalence of tensor categories, then
  $$
  \gamma(V)=\gamma(\CF(V))
  $$
  for all absolutely simple $V \in \Mod{H}$.
\end{thm}

When $H$ is pivotal, the 2nd Frobenius-Schur indicator $\nu_2(V)$ and pivotal dimension $\qdim(V)$ are defined for each finite-dimensional $H$-module $V$. In this case, we prove in Proposition
\ref{p:product} that
$$
\gamma(V) = \nu_2(V) \qdim(V)= \nu_2(V) \qdimr(V)
$$
for each absolutely simple $H$-module $V$.

To prove Theorem \ref{t:1}, we first need the following lemma.
 We thank H.-J. Schneider for helpful conversations about the lemma.

\begin{lem}\label{l:trace anti map}
Let $U=[U_{ij}]$ be an invertible element of $M_n(\k)$, and $A$ an algebra anti-automorphism of $M_n(\k)$ defined by $A(X)=UX^tU\inv$ for some $U \in \GL_n(\k)$, where $X^t$ denotes the transpose of $X$. Then
$$
\Tr(A)=\tr(U^tU\inv)
$$
where $\tr$ is the ordinary trace of matrices.
\end{lem}
\begin{proof}
  Let $E_{ij}$ be the matrix $[\delta_{ij}] \in M_n(\k)$, and $U\inv=[\ol U_{ij}]$. Since $\langle X, Y \rangle=\tr(XY^t)$ defines a non-degenerate symmetric bilinear form on $M_n(\k)$ and $\langle E_{ij}, E_{kl} \rangle =\delta_{ij, kl}$,  we have
  $$
   \Tr(A)  = \sum_{i,j} \langle A(E_{ij}), E_{ij}\rangle= \sum_{i,j} \tr(UE_{ji}U\inv E_{ji}) = \sum_{i,j} U_{ij} \ol U_{ij} =\tr(U^t U\inv)\,. \qedhere
  $$
\end{proof}
The second step is to show that $\gamma(V)$ is invariant under 2-cocycle twisting of $H$. Let $F$ be a 2-cocycle of $H$, and $V \in \Mod{H}$. We denote by $V_F$ the same $H$-module $V$ but considered as an object in $\Mod{H^F}$.
\begin{prop}\label{p:3}
Let $H$ be a finite-dimensional Hopf algebra over $\k$ and $F$ a 2-cocycle of $H$. Then
$$
\gamma(V) = \gamma(V_F)
$$
for all absolutely simple $V \in \Mod{H}$.
\end{prop}
\begin{proof}
 Let $F=\sum_i f_i \o g_i$ be a 2-cocycle of $H$ and $F\inv=\sum_j d_i\o e_j$.
Then the twist $H^F$ of $H$ by $F$ has the antipode $S^F$ given by
$$
S^F(h)=uS(h) u\inv
$$
where $u=\sum_i f_iS(g_i)$ and $u\inv = \sum_i S(d_j)e_j$. Suppose $V$ is an absolutely simple object of $\Mod{H}$. The left dual $V\du$ in $\Mod{H}$ is different from $V_F\du$ in $\Mod{H^F}$, but they are isomorphic as $H$-modules under the duality transformation (cf. \cite{NS07a}) $\tilde\xi: V\du \to V\du_F$ defined by
$$
\tilde \xi (f)(v)=f(u\inv v)\quad \text{ for $f \in V^*$ and $v \in V$.}
$$
Thus, $V$ is self-dual in $\Mod{H}$ if, and only if, $V_F$ is self-dual in $\Mod{H^F}$. Therefore,
$\gamma(V)=\gamma(V_F)=0$ if $V$ is not self-dual in $\Mod{H}$.

Now, let us further assume $V$ is self-dual in $\Mod{H}$. Since $V_F=V$ as $H$-modules,
$$
H/\ann V = H/\ann V_F \stackrel{\phi}{\cong} M_n(\k)
$$
as $\k$-algebras, where $n=\dim_\k V$. We write $\ol h$ for $\phi(h+\ann V) \in M_n(\k)$, and
let
\begin{equation}\label{eq:Sbar}
\ol S:=\phi \circ S_V \circ \phi\inv \quad\text{and}\quad \ol {S^F}:=\phi \circ S^F_{V_F} \circ \phi\inv\,.
\end{equation}
Obviously, both $\ol S$ and $\ol {S^F}$ are algebra anti-automorphisms on $M_n(\k)$, and
 $$
 \gamma(V) =\Tr(\ol S), \quad \text{and}\quad\gamma(V_F) = \Tr(\ol {S^F})\,.
 $$
 Moreover, there is an invertible matrix $U$ such that
 \begin{equation}\label{eq:U1}
\ol {{S}(h)} =  \ol {S}(\ol h) = U {\ol h}^t U\inv.
\end{equation}
From Equation \eqref{eq:U1} we get the following equalities:
\begin{equation}\label{eq:U2}
 \ol {{S^F}(h)} = \ol {S^F}(\ol h) = {\ol u} U {\ol h}^t U\inv  {\ol u}\inv\,,\quad
 \ol {S^2(h)} = \ol S^2(\ol h) = U(U\inv)^t {\ol h} U^tU\inv
\end{equation}
for all $h \in H$.

In view of Lemma \ref{l:trace anti map},
$$
\gamma(V)= \Tr(\ol S) = \tr(U^tU\inv)\,.
$$
 Thus,
by Equations \eqref{eq:U1}, \eqref{eq:U2} and Lemma \ref{l:trace anti map}, we find
\begin{multline*}
 \gamma(V_F) = \Tr(\ol {S^F})= \tr(U^t {\ol u}^t U\inv {\ol u}\inv) =\tr(U^t U\inv  \ol{S(u)u\inv}) \\
 = \sum_{i,j}\tr(U^t U\inv  \ol{S^2(g_i)S(d_jf_i)e_j})
 = \sum_{i,j}\tr(\ol{g_i} U^t U\inv \ol{S(d_jf_i)e_j})\\
 =\sum_{i,j}\tr(U^t U\inv \ol{S(d_jf_i)e_jg_i})
 =\tr(U^t U\inv)=\gamma(V)\,. \qedhere
\end{multline*}
\end{proof}

\begin{proof}[Proof of Theorem \ref{t:1}]
By Theorem \ref{t:p1}, if a $\k$-linear equivalence $\CF : \Mod{H} \to \Mod{H'}$ defines a tensor equivalence, then  $H$ and $H'$ are gauge equivalence, i.e. there exist a 2-cocycle $F$ of $H$ and a bialgebra isomorphism $\sigma: H' \to H^F $. Moreover, $\CF$ is naturally isomorphic to the $\k$-linear equivalence $\lsub \sigma (-) : \Mod{H} \to \Mod{H'}$ induced by the algebra isomorphism $\sigma$. For any absolutely simple $V \in \Mod{H}$,  $V$ is self-dual in $\Mod{H}$ if, and only if, $\CF(V)$ is self-dual in $\Mod{H'}$. Thus, if $V$ is not self-dual, then $\gamma(V)=\gamma(\CF(V))=0$.

Assume $V$ is self-dual in $\Mod{H}$. Then $\ann \lsub \sigma V = \sigma\inv(\ann V_F)$, and so
$\sigma$ induces an algebra isomorphism $\ol \sigma: H'/\ann \lsub \sigma V \to H/\ann V_F$.
Let $S, S'$ and $S^F$ denote the antipodes of $H$, $H'$ and $H^F$ respectively. Then
$$
\sigma\inv  \circ S^F \circ \sigma = S',
$$
and so
$$
\quad \ol \sigma\inv \circ S^F_{V_F} \circ \ol \sigma = S'_{\lsub \sigma V}\,.
$$
Therefore, by Proposition \ref{p:3},
$$
\gamma(\CF(V))=\gamma(\lsub{\sigma} V)= \Tr(S'_{\lsub \sigma V}) = \Tr(S^F_{V_F}) = \gamma(V_F)  =\gamma(V)\,. \qedhere
$$
\end{proof}

It is clear that the invariant $\gamma(V)$ of the tensor category $\Mod{H}$ is closely related to the 2nd Frobenius-Schur indicator when $H$ is split semisimple. A more general relationship among the invariant $\gamma(V)$, the 2nd Frobenius-Schur indicator $\nu_2(V)$ and the pivotal dimension $\qdim(V)$ also appears in the case of finite-dimensional pivotal Hopf algebras.

\begin{df}
  A \textbf{pivotal element} of a Hopf algebra $H$ with the antipode $S$ is group-like element $g$ such that $S^2(h) = ghg\inv$ for all $h \in H$. A Hopf algebra which admits a pivotal element is called \textbf{pivotal}.
\end{df}

Recall that there are two natural maps, the evaluation map $\ev: V\du \o V\to \k$ and the dual basis or coevaluation map $\db: \k \to V\o V\du$, associated with each finite-dimensional module $V$ over a Hopf algebra $H$, where $\{v_i\}$ is a basis for $V$ and $\{v^i\}$ is its dual basis in $V\du$.

Now, we assume $H$ is a pivotal Hopf algebra with a pivotal element $g$ and antipode $S$. Suppose $V \in \Mod{H}$. Then the map $j: V \to V\bidu$ defined by
$$
j(v)(f) = f(gv)\quad \text{for all $v \in V$ and $f \in V^*$},
$$
is a pivotal structure of $\Mod{H}$.  The (left) pivotal (or quantum) dimension $\qdim(V)$ of $V$ is defined as the scalar corresponding to the $\k$-linear map
$$
\k \xrightarrow{\db} V\du\o V\bidu  \xrightarrow{\id \o j\inv} V\du\o V \xrightarrow{\ev} \k\,.
$$
Direct simplification shows that
$$
\qdim(V)=\chi_V(g\inv)
$$
where $\chi_V$ is the character of $V$. The right pivotal dimension $\text{qdim}_r(V)$ can be defined similarly, and $\text{qdim}_r(V) = \chi_V(g)$.

Let us identify $\Hom_H(\k, V \o V)$ with the $H$-invariant space $(V\o V)^H$.  Following \cite{NS07b}, the map $E_2$ is defined by
\begin{equation}\label{eq:E2}
E_2(\sum_i u_i \o v_i) = \sum_i v_i \o g\inv u_i\quad \text{for all }\sum_i u_i \o v_i \in (V\o V)^H,
\end{equation}
and
\begin{equation}\label{eq:nu2}
\nu_2(V) = \Tr(E_2)\,.
\end{equation}
Note that $E_2^2 =\id$ and $\Hom_H(\k, V \o V) \cong \Hom_H(V\du, V)$ as $\k$-linear spaces.
Thus, if  $V$ is \emph{absolutely} simple, then
$$
\dim (V\o V)^H= \left\{
\begin{array}{ll}
1 & \text{if } V \cong V\du,\\
0 & \text{otherwise}\,.
\end{array}
\right.
$$
Therefore,  $\nu_2(V)=\pm 1$ if $V \cong V\du$, and $0$ otherwise. In particular, $\Tr(E_2)$ is the eigenvalue of $E_2$ when $V$ is self-dual.

If $V$ is absolutely simple and self-dual in $\Mod{H}$, then $\dim \Hom_H(\k, V\du \o V\du)=1=\dim \Hom_H(V, V\du)$. Let $f \in \Hom_H(V, V\du)$ be a non-zero element. Then
$$
\bb(x, y) := f(x)(y) \quad \text{for all }x, y \in V,
$$
 defines an $H$-invariant non-degenerate bilinear form on $V$. Conversely, if $\bb'$ is a non-zero $H$-invariant bilinear form on $V$, then $f'(x)=\bb'(x,y)$ defines a non-zero $H$-module map from $V$ to $V\du$. Thus, $f'$ is a scalar multiple of $f$ and so $\bb'$ is a scalar multiple of $\bb$.

Note that $\bb=\sum_i u_i^* \o v_i^* \in V\du \o V\du$, and the assignment $1 \mapsto \sum_i u_i^* \o v_i^* \in V\du \o V\du$ defines a non-zero map in $\Hom_H(\k, V\du \o V\du)$. By \eqref{eq:E2} and \eqref{eq:nu2}, we find
$$
\nu_2(V) \sum_i u_i^* \o v_i^* = \sum_i v_i^* \o g\inv u_i^*\,.
$$
In terms of $\bb$, we have the relation
\begin{equation}\label{eq:nu}
\nu_2(V) \bb(x,y) = \bb(y, gx) \quad\text{for all }x,y \in V\,.
\end{equation}
These paragraphs have summarized the Frobenius-Schur Theorem for  absolutely simple self-dual modules over a finite-dimensional pivotal Hopf algebra (cf. \cite{LM} and \cite{MaN}).
\begin{prop}\label{p:product}
  Let $H$ be a finite-dimensional Hopf algebra over $\k$ with antipode $S$. If $H$ admits a pivotal element $g$, then
  \begin{equation}\label{eq:gamma2}
    \gamma(V)=\nu_2(V)\cdot \qdim(V) =\nu_2(V)\cdot \qdimr(V)  \,.
  \end{equation}
  for all absolutely simple $V \in \Mod{H}$, where $\nu_2(V)$, $\qdim(V)$ and $\text{qdim}_r(V)$ are computed using the pivotal element $g$. In particular, $\nu_2(V) \qdim(V)$ is independent of the choice of the
pivotal element $g$.

\end{prop}
\begin{proof}
  Let $V$ be an absolutely simple $H$-module. If $V$ is not self-dual, then  $\gamma(V)=\nu_2(V)=0$ and so the equalities hold trivially. Assume $V$ is self-dual, and let $\{v_1, \dots, v_n\}$ be a basis for $V$. Then $hv_j = \sum_{j=1}^n h_{ij}v_i$ for some $h_{ij} \in \k$, and we write $\ol h$ for the matrix $[h_{ij}]$. The assignment $\pi: h \mapsto \ol h$ defines a matrix representation of $H$ afforded by $V$ with $\ker \pi = \ann V$. Since $V$ is absolutely simple, $\pi$ is surjective, and hence $\pi$
  induces an algebra isomorphism $\phi : H/\ann V \to M_n(\k)$ such that $\phi (h +\ann V)=\ol h =\pi(h)$ for all $h \in H$.

  Following the notation in the proof of Proposition \ref{p:3}, we let $\ol S = \phi \circ S_V \circ \phi\inv$. Then
  $$
  \ol S(X) = U X^t U\inv, \quad \text{and so} \quad \ol S^2(X) = (U^tU\inv)\inv X U^t U\inv
  $$
  for some $U \in \GL_n(\k)$. On the other hand,
  $$
  \ol S^2(X) = \ol g X \ol  {g\inv}.
  $$
  Therefore, $U^t U\inv \ol g$ is in the center of $M_n(\k)$ and hence
  $$
  U^t U\inv \ol g=c I\quad\text{for some } c \in \k.
  $$
   Thus, by Lemma \ref{l:trace anti map},
  $$
  \gamma(V) = \tr(U^tU\inv) = c \tr(\ol{g\inv})=c \chi_V(g\inv)\,.
  $$
  We claim that $c =\nu_2(V)$. Let $\{e_1, \dots, e_n\}$ denote the standard basis for $\k^n$ and let $(\cdot,\cdot)$ denote the standard non-degenerate symmetric bilinear form on $\k^n$.
  We define  the bilinear form $\bb$ on $V$ by extending the assignment $\bb(v_i, v_j) = (e_i, U\inv e_j)$ linearly. Then for $h \in H$,
  \begin{align*}
  \bb(h_1 v_i, h_2 v_j) &= (\ol{h_1} e_i, U\inv \ol{h_2} e_j)= (e_i,\ol{h_1}^t U\inv \ol{h_2}e_j)
  = (e_i, U\inv \ol S(\ol{h_1})\ol{h_2} e_j)\\
  &=(e_i, U\inv \ol{S(h_1)h_2} e_j) =\eps(h)(e_i, U\inv e_j) = \eps(h)\bb(v_i,v_j)\,.
  \end{align*}
  Therefore, $\bb$ is a non-zero $H$-invariant form on $V$. Moreover,
  \begin{align*}
  \bb(v_j, g v_i)&= (e_j, U\inv \ol g e_i)=(e_j, (U\inv)^t U^t U\inv \ol g e_i) \\
  &= c(e_j, (U\inv)^t  e_i) = c(U\inv e_j,  e_i) = c \bb(v_i, v_j)\,.
  \end{align*}
  It follows from \eqref{eq:nu} that $c=\nu_2(V)$.

  Let $\{u_i\}$ be the basis for $V$ such that $\bb(u_i, v_j)=\delta_{ij}$.  Then
  $$
  \qdimr(V) = \chi_V(g) = \sum_i \bb(gu_i, v_i) = \sum_i \bb(u_i, g\inv v_i) =\chi_V(g\inv) = \qdim(V)\,.
  $$
  The third equality is a consequence of the $H$-invariance of $\bb$.

  The left hand side of the first equation of \eqref{eq:gamma2} indicates the expression on the right hand side is independent of the choice of the pivotal element $g$ of $H$.
\end{proof}
\begin{remark}
For the quantum group $u_q(sl_2)$ at the primitive $n$th root of unity $q$ with $n$ odd, Jedwab has shown in \cite{J} that $\gamma(V)=(-1)^{\dim V+1} \qdim V$ for every simple module $V$ of dimension less than $n$. This result together with Proposition \ref{p:product} implies that $\nu_2(V) = (-1)^{\dim V+1}$ for $\dim V < n$.
\end{remark}

\section{On the values of the indicators}\label{values}
In Section \ref{taft}, we have seen that the value of $\nu_2(T)$ for $T$ a Taft algebra over $\BC$ is a complex number. However, if $H$ is a semisimple Hopf algebra over $\BC$ with antipode $S$, then $\nu_2(H)=\Tr(S)$ is always an integer. This observation follows immediately by the Larson-Radford theorem \cite{LR}, $S^2 = \id_H$ which implies that all of the eigenvalues of $S$ are $\pm 1$.

\subsection{Positivity of $\Tr(S)$ for abelian extensions of Hopf algebras}
 For a group algebra $H=\BC G$,  it is easy to see that $\Tr(S)$ is equal to the number $i_G$ of \emph{involutions} in $G$. Therefore, $\nu_2(H)=i_G \ge 1$. More generally, we have the following observation.
 \begin{prop}\label{known}
 Let $H = \BC^G\#\BC F$ be a bismash product determined by the factorizable group
$L = FG.$ Then $\Tr(S_H) = \Tr(S_{\BC L}) = i_L.$ In particular, if $H = D(\BC G)$, then $\Tr(S_H) =  i_G^2.$
\end{prop}

\begin{proof} The first part is \cite[Lemma 2.8]{JM}. In the case of $D(\BC G)$,
$L = G \rtimes G$, with $G$ acting on itself by conjugation, and it is easy to see that $i_L = i_G^2.$
This is noted in \cite{GM}.
\end{proof}

From this proposition, one might hope that if $H$ is semisimple over $\mathbb C$,
then $\Tr(S)$ is always positive. However, this is not the case, even for group-theoretical Hopf algebras.
\begin{ex} {\rm Let $H=\BC^{G}\#_{\sigma }^{\tau }\BC F$ where $G=\langle x\rangle \times \langle y\rangle$ and  $F=\langle t \rangle$ are multiplicative groups isomorphic to $\BZ_4 \times \BZ_4$ and $\BZ_2$ respectively, and the coaction $\rho$ and the cocycle
$\sigma: F \times F \to \BC^G$ are trivial. Let $\{p_g\mid g \in G\}$ denote the dual basis of $G$ for $\BC^G$. The $F$-action
$\rightharpoonup$ on $G$, and dual cocycle $\tau: F \to \BC^G \ot \BC^G $ are
defined via
$$
t \rightharpoonup g =g^{-1}  \text{ for all }g\in G, \quad \text{ and }\quad
\tau \left( t\right)  =\sum_{g, h \in G}\tau _{t}(g,h) p_g \ot
p_h
$$
where $\tau_t(x^iy^j, x^k y^l) = (-1)^{ik+jl+jk}$.
One can verify directly that $\tau_t : G \times G \to \BC^\times$ is a 2-cocycle on $G$, that is, $\tau_t$ satisfies the functional equation $$ \tau_t(a,b)\tau_t(ab, c) = \tau_t(a, bc)\tau_t(b,c) \text{ for all }a,b,c \in G.
$$
which makes $\tau$ a dual cocycle (see \cite[Proposition 2.16]{AD} or \cite[Lemma 4.5]{KMM} for the particular case of cocentral abelian extensions). Moreover, the equality $$ \tau_t(a,b)\tau_t(t\rightharpoonup a, t \rightharpoonup b)=1 =\tau_1(a,b) $$ holds for all $a, b \in G$ and therefore comultiplication in $H$ is multiplicative. Then it follows from \cite[Theorem 2.20]{AD} that $H$ is a Hopf algebra. The antipode $S$ of $H$ is given by
$$
S(p_g \# z) = {\tau_z(g\inv, g)}\inv p_{z\inv \rightharpoonup g\inv}\# z\inv =
{\tau_z(g\inv, g)} p_{z \rightharpoonup g\inv}\# z\,.
$$
for $z \in F$ and $g \in G$. Note that $z \rightharpoonup g\inv =g$ if, and only if, $z=t$ or $g^2=1$. Therefore,
$$
\Tr(S) = \sum_{g \in G} \tau_t(g, g^{-1})+\sum_{\substack{g \in G \\ g^2=1}} \tau_1(g, g^{-1}) \,.
$$
It is easy to see that $x^2, y^2, x^2 y^2$ and $1$ are all of the involutions of $G$, and thus
$\sum\limits_{\substack{g \in G \\ g^2=1}} \tau_1(g, g^{-1})=4$. Therefore,
\begin{align*}
\Tr(S) &= \sum_{g \in G} \tau_t(g, g^{-1})+4 = \sum_{i,j=0}^3 \tau_t(x^i y^j, x^{-i} y^{-j})+4\\
&= \sum_{i,j=0}^3 (-1)^{-i^2-j^2-ij} +4 = 4-12+4 =-4
\end{align*}
since
\begin{equation*}
\left( -1\right) ^{i^{2}+j^{2}+ij}=\left\{
\begin{array}{ll}
1 & \text{ if }i\text{ and }j\text{ are both even,} \\
-1 & \text{ otherwise.}%
\end{array}%
\right. \qed
\end{equation*}
}\end{ex}
It is still interesting to know just when $\Tr(S)$ is positive. We
 give some criteria for a cocentral abelian extensions to have
$\Tr(S)>0$ in the following proposition.
\begin{prop} \label{p:5.3}
  Let $G$, $F$ be finite groups and $H=\BC^{G}\#_{\sigma }^{\tau }\BC F$ a
cocentral abelian extension with antipode $S$.
\begin{enumerate}
  \item[\rm (i)] If $|F|$ is odd, then $\Tr(S) = i_G$.
  \item[\rm (ii)] Let $I_F$ denote the set of all involution in $F$, $F_g$ the stabilizer of $g$ and
      $$
      \widetilde {F_g} = \{ w\in F \mid w\rightharpoonup g=g^{\pm 1}\}.
       $$
       Assume that the cocycle $\sigma$ is trivial, the set $I_F$ is a subgroup of $F$, and $F_g = \widetilde {F_g}$ for all $g \in G$. Then $\Tr(S) > 0$.
\end{enumerate}
\end{prop}
\begin{proof}
  Recall that
$$
\Tr(S) =\sum_{\substack{(g,w)\in G \times F \\ w^2=1,\, w\rightharpoonup g =
g^{-1}}}
\sigma_g^{-1}(w,w)\tau_{w}^{-1}(g,g^{-1})\,.
$$
(i) If $|F|$ is odd, then we have
$$
  \Tr(S)= \sum_{g \in G,\, g = g^{-1}}
\sigma_g^{-1}(1,1)\tau_{1}^{-1}(g,g^{-1}) = \sum_{g\in G,\, g^2=1
} 1 = i_G.
$$
(ii) Let $g\in G$. Define $\mu _g : \BC F_g \to \BC$ via $\mu _g (w)
=\tau_{w}(g,g)$ for $w\in F_g$. Then by \cite[Lemma 4.5]{KMM}, since
$\sigma$ is trivial and $\widetilde {F_g} = F_g$, $\mu _g$ is a
one-dimensional character of $\BC F_g$. Then
$$
\Tr(S) =\sum_{\substack{(g,w)\in G \times F \\ w^2=1,\, w\rightharpoonup g =
g^{-1}}} \tau_{w}^{-1}(g,g^{-1})
=\sum_{\substack{(g,w)\in G \times F \\ w^2=1,\,g^2=1,\, w\rightharpoonup g =
g}}\tau_{w}^{-1}(g,g)
= \sum_{g\in I_G }\sum_{w\in I_F \cap F_g } \mu_{g}^{-1}(w)\,.
$$
By the orthogonality of group characters, we find
\begin{equation*}
 \sum_{w\in I_F \cap F_g } \mu_{g}^{-1}(w) = \delta |I_F \cap F_g|
\end{equation*}
where $\delta=1$ if $\mu_g|_{I_F \cap F_g} = 1$, and $\delta =0$ otherwise.
Therefore, $\Tr(S) > 0$.
\end{proof}
\begin{remark}
{\rm\mbox{\hspace{1cm}}
\begin{enumerate}
\item[(i)] If $F$ is abelian then $I_F$ is a subgroup of $F$.
\item[(ii)]  If $F$ is cyclic then we may assume that the cocycle $\sigma$ is trivial (since the group $H^2 (F,(\BC^{G})^{\times})$ is trivial).
\item[(iii)]  In particular, if $F$ is cyclic and $G$ is an elementary abelian 2-group then the conditions of Proposition  \ref{p:5.3} are satisfied and therefore $\Tr(S) > 0$.
\end{enumerate}
}
\end{remark}
\subsection{Positivity of the indicators of modular quasi-Hopf algebras}
 A semisimple quasi-triangular quasi-Hopf algebra $H$ over $\BC$ is said to be \emph{modular} if the braided spherical fusion category $\Mod{H}$ is a modular tensor category (cf. \cite{ENO} and \cite{BaKi}).  By \cite{Mu}, $\Mod{D(H)}$ is  modular for any semisimple quasi-Hopf algebra $H$ over $\BC$, where $D(H)$ denotes the quantum double (or Drinfeld double) of $H$. Note that every semisimple factorizable Hopf algebra $H$ over $\BC$ is modular (cf. \cite{Tak01}).

The $n$-th Frobenius-Schur indicator $\nu_n(V)$ of an object $V$  in a spherical fusion category $\CC$ over $\BC$ is defined in \cite[p 71]{NS07b}. In addition, if $\CC$ is modular, some canonical linear combination of indicators are always real and non-negative.
\begin{prop}\label{p1}
Let $\CC$ be a modular tensor category over $\BC$, and $U_0, \dots, U_\ell$
a complete list of non-isomorphic simple objects of $\CC$ with $U_0=I$, the unit object.
Then
$$
\sum_i d_i \nu_n(U_i) \ge 0
$$
for all positive integers $n$, where $d_i$ and $\nu_n(U_i)$ respectively denote the pivotal dimension and
the $n$-th Frobenius-Schur indicator of $U_i$.
\end{prop}

\begin{proof} Let $\theta$ be the ribbon structure of the modular category $\CC$.
By \cite[Theorem 7.5]{NS07a},
  $$
  \nu_n(U_k) = \frac{1}{\dim \CC} \sum_{i,j} N_{ij}^k d_i d_j \frac{\om_i^n}{\om_j^n}
  $$
  where $N_{ij}^k = \dim \CC(U_k, U_i \ot U_j)$ and $\om_i \in \BC$ is given by $\om_i \id_{U_i} = \theta_{U_i}$. In particular, $\om_i$ is an $n$-th root of unity where $n$ is the order of $\theta$. Note that
  $\sum_k N_{ij}^k d_k =d_i d_j$, and $d_k \in \BR$ for all $k$ by \cite{ENO}.
  We have
  \begin{multline*}
    \sum_k d_k \nu_n(U_k) = \frac{1}{\dim \CC} \sum_{i,j, k} N_{ij}^k d_i d_j d_k \frac{\om_i^n}{\om_j^n} \\
    = \frac{1}{\dim \CC} \sum_{i,j}  d_i^2 d_j^2  \frac{\om_i^n}{\om_j^n}
    = \frac{1}{\dim \CC} \left|\sum_i d^2_i \om_i^n\right|^2 \ge 0\,,
  \end{multline*}
where in the last equation we have used that $\overline{\om_i} = \om_i\inv$. \qedhere
\end{proof}

\begin{thm} \label{t:pos}
  Let $H$ be a modular quasi-Hopf algebra over $\BC$. Then  for all positive integers $n$, the $n$-th Frobenius-Schur indicator $\nu_n(H)$ is real and non-negative. In addition, \begin{enumerate}
  \item[\rm (i)] if $H=D(A)$ for some semisimple quasi-Hopf algebra $A$ over $\BC$ then
  $$
  \nu_n(H) = |\nu_n(A)|^2 \ge 0\,.
  $$
\item[\rm (ii)] If $H$ is a semisimple factorizable Hopf algebra over $\BC$ with antipode $S_H$ then $\Tr(S_H) \ge 0$.
\item[\rm (iii)] If $H =D(A)$ for some semisimple Hopf algebra $A$ then $\Tr(S_H)=\Tr(S_A)^2$.
\end{enumerate}

\end{thm}

\begin{proof}
  Let $U_0, \dots, U_\ell$ be a complete list of non-isomorphic simple objects of $\CC=\Mod{H}$, where $U_0$ is the
  trivial $H$-module $\BC$. Then $H\cong \bigoplus_{i=0}^\ell d_i U_i$ where $d_i = \dim U_i$. By the additivity
  of indicators, we have
  $$
   \nu_n(H) = \sum_{i=0}^\ell d_i \nu_n(U_i) \,.
  $$
  It follows from the proof of Proposition \ref{p1} that
  \begin{equation} \label{eq:1}
  \nu_n(H)=  \left|\frac{1}{\sqrt{\dim H}}\sum_{i=0}^\ell d^2_i \om_i^n\right|^2 \ge 0
  \end{equation}
  for all $n \in \BN$, where $\om_i$ is the scalar of the ribbon structure component $\theta_{U_i}$.

  If $H=D(A)$ for some semisimple quasi-Hopf algebra $A$ over $\BC$, then by \cite[Theorem 4.1]{NS07a},
  $$
  \nu_n(A) = \frac{1}{\dim A} \Tr(\theta^n_{D(A)\ot_A A}) = \frac{1}{\dim A} \sum_{i=0}^\ell d_i^2 \om^n_i\,.
  $$
  It follows from \eqref{eq:1} that
  $$
  \nu_n(H) = |\nu_n(A)|^2\,.
  $$

  If $H$ is a semisimple factorizable Hopf algebra over $\BC$, then, by \cite[Theorem 3.1]{LM},
  $$
  \Tr(S_H)=\nu_2(H) \ge 0\,.
  $$
  Moreover, if $H=D(A)$ for some semisimple Hopf algebra $A$, then $\Tr(S_H)$ and $\Tr(S_A)$ are integers, and so we have
  $$
  \Tr(S_H) = \nu_2(H) = |\nu_2(A)|^2 = |\Tr(S_A)|^2=\Tr(S_A)^2\,.\qedhere
  $$
\end{proof}

We note that when $A = \BC G$ and $H=D(A)$, we already saw in Proposition \ref{known} that $\Tr(S_H) = i_G^2 = \Tr(S_A)^2.$ In general, we can ask the following question.\\

\noindent
{\bf Question:} For any  finite-dimensional complex Hopf algebra $H$, is it true that $\nu_n(D(H))=|\nu_n(H)|^2$ for all positive integer $n$?

\subsection{Prime divisors of the dimension of a semisimple quasi-Hopf algebra} It is well-known that a group $G$ has even order if, and only if, $i_G \ne 1 $. If $S$ denotes the antipode of the group algebra $H=\BC G$, then $\dim \BC G$ is even if, and only if, $\nu_2(H)  = \Tr(S) =i_G \ne 1$. This observation holds for every semisimple Hopf algebra over $\BC$.
\begin{prop} \label{p:5.6}
  Let $H$ be a semisimple Hopf algebra over $\BC$. Then $\nu_2(H) \ne 1$ if, and only if, $\dim H$ is even.
\end{prop}
\begin{proof}
It is proved in \cite{KSZ} that $\dim H$ is odd if, and only if, the  trivial $H$-module $V_0$ is the only self-dual simple $H$-module. Therefore,  if $\dim H$ is odd, $V_0$ is the only simple $H$-module $V$ with $\nu_2(V) \ne 0$. By \cite{LM}, we have
\begin{equation}\label{LM1}
 \Tr(S)=\nu_2(H)= \sum_{V \text{simple}} \dim V \cdot \nu_2(V)
\end{equation}
which implies
$\nu_2(H)=\dim (V_0)\cdot \nu_2( V_0 ) = 1$. Conversely, assume $\dim(H)$ is even. As noted above, any eigenvalue of $S$ is 1 or -1. Suppose $a$ of them are $1$ and $b$ of them are $-1$. Then $a+b=\dim H$ is even, so $a$ and $b$
have the same parity and therefore $\nu_2(H)=\Tr (S) = a-b$ is also even. Thus $\nu_2(H) \neq 1$.
\end{proof}

Using some recent results of \cite{NS07a},  the preceding proposition can be generalized to any prime number $p$ for any semisimple quasi-Hopf algebra over $\BC$.
\begin{thm} \label{t:primeind}
 Let $H$ be a finite-dimensional semisimple quasi-Hopf algebra over $\BC$ and $p$ a prime number. Then the following statements are equivalent:
 \begin{enumerate}
 \item[(i)] $p \mid \dim H$.
 \item[(ii)] $\nu_p(V) \ne 0$ for some non-trivial simple $H$-module $V$.
 \item[(iii)] $\nu_p(H) \ne 1$.
 \end{enumerate}
\end{thm}

\begin{proof}
 By \cite{Mu}, the quantum  (Drinfeld) double $D(H)$ of $H$ is modular. Let $\theta$ be the ribbon structure on $D(H)$ associated with
  the canonical pivotal structure, $N$ the order of $\theta$, and $\zeta_N$ a primitive $N$-th root of unity.

  (ii) $\Rightarrow$ (i): If $p \nmid \dim H$, then $p \nmid N$ (cf. \cite[Theorems 8.4 or 9.1]{NS07a} or \cite{Etingof02}). Therefore, $\sigma:\zeta_N \mapsto \zeta_N^p$ defines
  an automorphism of $\BQ(\zeta_N)/\BQ$. Moreover, by \cite[Theorem 4.1]{NS07a},
 \begin{equation}\label{eq:Tr}
    \nu_p(V) = \frac{1}{\dim H} \Tr(\theta^p_{K(V)})
 \end{equation}
  for $V \in \Mod{H}$, where $K(V) = D(H) \ot_H V$. Let $\hat{\G}=\{U_0, \dots,U_\ell\}$
  be a complete set of non-isomorphic simple $D(H)$-modules, and $\om_i$ be the scalar of the
  component $\theta_{U_i}$. Then $\om_i$ is a power of $\zeta_N$, and so $\sigma(\om_i) =\om_i^p$. Now
  let $V \in \Mod{H}$ be simple. One can consider the action of the Galois group on the indicators (cf. \cite[p 24]{KSZ2} or \cite[p 62]{NS07a}),
  $$
  \nu_p(V) = \frac{1}{\dim H} \Tr(\theta^p_{K(V)}) = \frac{1}{\dim H} \sum_{i=0}^\ell N_i d_i \om^p_i = \sigma\left(\frac{1}{\dim H}\sum_{i=0}^\ell N_i d_i \om_i\right)=\sigma(\nu_1(V))
  $$
  where $N_i = \dim \Hom_{D(H)}(K(V), U_i)$, $d_i = \dim U_i$ . Note that
  $\nu_1(V)$ is equal to 0 if $V$ is not the unit object and 1 otherwise.
  Thus, $\nu_p(V)=0$ for all non-trivial simple $H$-modules $V$.

  (iii) $\Rightarrow$ (ii):  Let $V_0, V_1, \dots, V_n$ form a complete set of non-isomorphic simple $H$-modules with
  $V_0$ being the trivial $H$-module. Since $H = \bigoplus_{i=0}^n (\dim V_i) V_i$,
  $\nu_p(H) = 1+\sum_{i=1}^n (\dim V_i)\nu_p(V_i)$. Thus, $\nu_p(H) \ne 1$ implies that
  $\nu_p(V_i) \ne 0$ for some $i > 0$.

  (i) $\Rightarrow$ (iii): Suppose $\nu_p(H)=1$. Recall from \cite[p 71]{NS07b} that
   $\nu_p(H)$ is the ordinary trace of a $\BC$-linear automorphism $E$ on $\Hom_H(\BC, H^{\ot p})$
   and
   $E^p=\id$. Therefore, by a linear algebra argument (cf. \cite[p 26]{KSZ2}),
   $$
  \nu_p(H)=\Tr(E) \equiv  \dim \Hom_H (\BC, H^{\ot p}) \,\text{ mod  }p\,.
   $$
   Since $\dim \Hom_H (\BC, H^{\ot p}) =(\dim H)^{p-1}$, we have
   $$
   1 \equiv (\dim H)^{p-1} \,\text{ mod  }p
   $$
   Therefore, $p \nmid \dim H$.
\end{proof}
\begin{remark}
 In the case of semisimple Hopf algebras $H$, the  indicator was defined as $\nu_n(V)=\chi_V(\Lambda^{[n]})$. In that case $N=\exp(H)$ and $\exp(H)\mid (\dim H)^3$ \cite{EG0}. It was shown in \cite{KSZ2} that there exists a linear operator $E$ on $\Hom_H(\BC, V^{\o n})$ such that $E^n=\id$,  $\Tr(E)=\nu_n(V)$, and Equation \ref{eq:Tr} holds, where now $\theta^{\pm 1}$  is the Drinfeld element of $D(H)$. Then Theorem \ref{t:primeind} can be proved in the context of semisimple Hopf algebras with the same arguments as above,  replacing some facts on quasi-Hopf algebras with corresponding results established in \cite{KSZ2} and \cite{EG0}.
\end{remark}

\section{$\Tr(S)$ and the degrees of representations}

According to \cite[p 54]{I} and \cite[p 278]{JL}, one application of Frobenius-Schur
indicators for finite groups is to give an easier proof of the Brauer-Fowler theorem.
This theorem states that for a given positive integer $n$, there exist only finitely many simple
groups $G$ containing an involution $x$ such that
the centralizer of $x$ has order n; this was useful in the classification of finite simple
groups. Thus it seems worthwhile to try to find Hopf algebra analogs of these methods.

The main preliminary step in both \cite{I} and \cite{JL} shows that $H$ has a non-trivial
representation whose degree is bounded by a function involving $\Tr(S)$. This step generalizes
to Hopf algebras, provided we replace $i_G$ in the statement for $\BC G$ by $\Tr(S)$ in our result for $H$.

\begin{thm}\label{bound}
Let $H$ be a semisimple Hopf algebra of even dimension $n$ over $\BC$ and let $S$ be
its antipode. Then there exists an irreducible character $\chi= \chi^* \neq \eps$ such that
$$\deg(\chi) \leq \alpha := \frac{n-1}{|\Tr(S) -1|}.$$
Moreover if $\Tr(S) >1$ then $\nu_2(\chi)=1$, and if $\Tr(S) <1$ then $\nu_2(\chi)=-1$.
\end{thm}
\begin{proof}
The formula is obviously well-defined since $\Tr(S) \ne 1$ by Proposition \ref{p:5.6}. Also note that $n= \dim H = 1+\sum\limits_{\chi \neq \eps}  \chi(1_H) ^2$, where as before
$\chi$ runs through the set $\Irr(H)$ of irreducible characters of $H$. Let
$$
\CT = \{ \chi \in \Irr(H)\mid\nu_2(\chi)=1\text{ and }\chi \neq \eps\}, \quad \text{and}\quad
 \CT' = \{\chi \in \Irr(H) \mid \nu_2(\chi) =-1\}\,.
$$
Obviously, $\eps \notin \CT'$ since $\nu_2(\eps)=1$.
From \eqref{LM1},  $\Tr(S) = \sum\limits_{\chi} \nu_2(\chi)
\chi_2(1_H)$. Therefore, since $\nu(\chi)\in \{0,1, -1\}$,
\begin{equation}\label{trace}
\Tr(S) -1 = \sum_{\chi \in \CT} \chi(1_H) - \sum_{\chi \in \mathcal{T'}} \chi(1_H).
\end{equation}

Since $\Tr(S) \neq 1$, there are two possibilities: either $\Tr(S) >1$ or $\Tr(S) <1$.
First assume that $\Tr(S) >1$. Then by Equation ($\ref{trace}$)
$$ 0<\Tr(S) -1 \leq \sum\limits_{\chi \in \CT} \chi(1_H) $$
and, in particular,
$\CT$ is not empty. Thus, by the Cauchy-Schwartz inequality, we have
$$
( \Tr(S) -1) ^2 \leq (\sum_{\chi\in \CT} \chi(1_H)) ^2 \leq
|\CT| \sum_{\chi \in \CT} (\chi(1_H)) ^2 \leq |\CT| (n-1)
$$
which implies
$$
\sum_{\chi \in \CT} ( \chi(1_H)) ^2
\leq \sum_{\chi\neq \eps}( \chi(1_H)) ^2 =n-1
\leq |\CT| \frac{(n-1)^2}{ (\Tr(S) -1)^2}.
$$
Therefore there exists $\chi \in \CT$ such that $\chi(1_H)
\leq \ds\frac{n-1}{\Tr(S) - 1}$.

The case that $\Tr(S) <1$ is similar.  By
 \eqref{trace}, we obtain
$$
0<1-\Tr(S)  \leq \sum_{\chi\in \CT'}\chi(1_H)
$$
and, in particular, $\mathcal{T'}$ is non-empty. By the same argument, we have
$$
\sum_{\chi \in \mathcal{T'}} ( \chi(1_H)) ^2  \leq n - 1 \leq |\CT'| \frac{(n-1)^2}{( 1- \Tr(S))^2}.
$$
which implies $\chi(1_H) \leq \dfrac{n-1}{1-\Tr(S) }.$
\end{proof}
\begin{remark}
  In view of Theorem \ref{t:primeind}, $\nu_2(H) \ne 1$ for any semisimple quasi-Hopf algebra $H$ of even dimension. By the same proof, Theorem \ref{bound} remains to hold for any even dimensional semisimple quasi-Hopf algebra over $\BC$ if one replaces $\Tr(S)$ by $\nu_2(H)$.
\end{remark}
\noindent{\bf Acknowledgement:}
  The third author would like to thank Susan Montgomery and the University of Southern California for their hospitality during his sabbatical leave.


\end{document}